\documentclass[12]{amsart}
\usepackage{latexsym}
\usepackage{amsthm}
\usepackage{amssymb}
\usepackage{amsfonts, amsmath}
\newtheorem{lm}{Lemma}

\newtheorem*{Thm}{Theorem}

\begin{document}
\title{ On q-convex with corners functions in complex manifolds and extension of line bundles}
\author{}
\noindent
%\begin{document}
\maketitle
%\begin{center}
%\noindent}
\begin{center}
\em{Youssef Alaoui}\\
\em{y.alaoui@iav.ac.ma}\\
\end{center}

\noindent
\em{Department of Mathematics,
Hassan II Institute of Agronomy}\\ 
\em{and Veterinary Sciences,
Madinat Al Irfane, BP 6202, Rabat, 10101, Morocco,}\\
%\end{center}
%\subjclass{32E10, 32E40.}
%\documentclass{article}
%\usepackage{authblk}
%\title{An important paper}
%\author[$\dagger$]{}
%\begin{document}
%\maketitle
%\addresses[$\dagger$]{Department of Mathematics, Institut Agronomique et V\'et\'erianire Hassan II. B.P. 6202, Rabat-Instituts, 10101. Morocco}
%\email[Y.~ Alaoui]{comp5123ster@gmail.com}
%\keywords{Stein spaces, $q$-complete spaces, $q$-convex functions, linear sets,
%$1$-convex functions with respect to a linear set.}
\date{}
\linespread{1.3}
%\maketitle
\date{}
%\setlength{\oddsidemargin}{0.5cm} \setlength{\evensidemargin}{0cm}
%\textwidth 14cm \topmargin -2.5cm \headheight 3cm\headsep 2cm
%\textheight 21cm \footskip 2cm \pagestyle{headings}
%\linespread{1.3}
%\maketitle
\begin{abstract}
In this article, we prove that if $X$
is a complex manifold of dimension $n\geq 4$ such that
there exists a $q$-convex with corners function $f\in F_{q}(X)$, then every holomorphic line bundle over $\{f>c\}$
extends uniquely to $X$ if $1\leq q\leq n-3$. This generalizes a well-known result obtained in \cite{ref5} for $q$-complete 
with corners complex manifolds with a corresponding exhaustion function $f \in F_{q}(X)$, when $n \geq 3q$.
\end{abstract}
\maketitle
\section{Introduction}
It was shown in \cite{ref5} that
if $X$ is a $q$-complete with corners complex manifold of dimension $n\geq 3q$ with a corresponding exhaustion function $\phi$ and, if
$L$ is a holomorphic line bundle on $X\setminus K$, where $K=\{z\in X : \phi(z)\leq c\}$, then there exists a unique holomorphic
line bundle $\tilde{L}$ over $X$ such that $\tilde{L}|_{X\setminus K}=L$ when $1\leq q\leq n-2$. The extension is unique modulo isomorphism.\\
\hspace*{.1in}Chen ~\cite{ref4} has given a counter example to extension for a pseudoconvex domain $\Omega\subset\subset\mathbb{C}^{n}$ 
when the compact $K\subset \Omega$ is of general shape in any dimension $n\geq 2$.\\
\hspace*{.1in}It has been proved earlier by Ivashkovich ~\cite{ref7} that there exist a two
dimensional real analytic totally real manifold $\Omega$ in $\mathbb{C}^{2}$, a non-compact subset $K\subset \Omega$ 
and a line bundle $L$ on $\Omega\setminus K$ which cannot be extended holomorphically to $\Omega$, even under the assymption that the bundle
is topologically trivial.\\
\hspace*{.1in}The main purpose of this article is to generalize the above result obtained in ~\cite{ref5}
to complex manifolds $X$ of dimension $n\geq 4$ which are not necessarily $q$-complete with corners but on which there exist a $q$-convex with corners function $f\in F_{q}(X)$ with $1\leq q\leq n-3$. More precisely we prove :
\begin{Thm}
Let $X$ be a complex manifold of dimension $n\geq 4$ and $\xi_{0}\in X$. Suppose that there exists a $q$-convex
with corners function $f\in C^{0}(X )$ with $1\leq q\leq n-3$, and let $Y=\{z\in X:
f(z)>f(\xi_{0})\}$. Then the restriction map
$$H^{p}(X, {\mathcal{O}^{*}})\rightarrow H^{p}(Y, {\mathcal{O}^{*}})$$
is bijective if $p=0, 1, 2$\\
and injective if $p=3$.
\end{Thm}
It should be noticed that in the above theorem the function $f\in F_{q}(X)$ is not necessarily exhaustive on $X$
and that the first lemma used in this paper was already obtained by the author in a preprint published on arXiv ~\cite{ref2}.\\
\\
\hspace*{.1in}We start by recalling some definitions concerning $q$-convexity.\\
\\
\hspace*{.1in}Let $X$ be a complex manifold. Then it is known that
a function $\phi\in C^{\infty}(X)$ is $q$-convex if for every
point $z\in X$, the Levi form
$L_{z}(\phi; z)$ has at most $q-1$ negative or zero eingenvalues
on each tangent space $T_{z}X$, $z\in X$.\\
\hspace*{.1in}We say that $X$ is $q$-complete if there exists
a $q$-convex function $\phi\in C^{\infty}(X,\mathbb{R})$
which is exhaustive on $X$ i.e. $\{x\in X: \phi(x)<c\}$
is relatively compact in $X$ for any $c\in\mathbb{R}.$\\
\hspace*{.1in}The space $X$ is said to be cohomologically $q$-complete
if for every coherent analytic sheaf ${\mathcal{F}}$ on $X$ the cohomology groups $H^{r}(X, {\mathcal{F}})$ vanish for all $r\geq q$.\\
\hspace*{.1in} A function $f : X\rightarrow \mathbb{R}$ is called $q$-convex
with corners , if $f$ is continuous and for each $x\in X$, there are a neighborhood
$U$ of $x$ in $X$ and $q$-convex functions $\phi_{1},\cdots, \phi_{r}$
on $U$ with $f|_{U}=max(\phi_{1},\cdots, \phi_{r})$.\\
We denote by $F_{q}(X)$ the set of $q$-convex functions with corners on $X$.\\
\section{Proof of the theorem}
%\noindent
In the proof of our generalization of Hartog's extension of line bundles, we shall need the following result (~\cite{ref3}, Proposition $12$) :
\begin{lm}{Let $D$ be a domain in $\mathbb{C}^{n}$, $\xi\in D$,
and let $\phi\in C^{\infty}(D )$ be a q-convex
function. Then for any coherent analytic sheaf ${\mathcal{F}}$ on $D$
there exists a fundamental system of Stein
neighborhoods $U\subset D$ of $\xi$ such that if $Y=\{z\in D:
\phi(z)>\phi(\xi)\}$, then $H^{p}(Y\cap U, {\mathcal{F}})=0$ for
$0<p<n-q$ and $H^{0}(U,{\mathcal{F}})\rightarrow H^{0}(U\cap Y, {\mathcal{F}})$ is an isomorphism.}
\end{lm}
In addition of  the theory of Andreotti-Grauert ~\cite{ref3}, the proof of the theorem uses the spectral sequence of Grothendieck ~\cite{ref6}
and also the Mittag-Leffler theorem in cohomolgy.
\begin{lm}{Let $X$ be a complex manifold of dimension $n$, and let $\phi:
X\rightarrow \mathbb{R}$ be a smooth $q$-convex function $\phi$ on $X$. Let $\xi_{0}\in X$
and $X'_{c}=\{x\in X: \phi(x)>c\},$ where $c=\phi(\xi_{0})$.
Then for any coherent analytic sheaf ${\mathcal{F}}$
on $X$ the restriction map
$$H^{p}(X,{\mathcal{F}})\rightarrow H^{p}(X'_{c}, {\mathcal{F}})$$
is bijective if $p\leq n-q-1$,\\
\hspace*{.1in}and injective if $p=n-q.$}
\end{lm}

\noindent
\begin{proof}
Let $V\subset\subset X$ be an open neighborhood of $\xi_{0}$ biholomorphic to a domain in $\mathbb{C}^{n}$.
Then there exists, by theorem $1$, a fundamental system of connected Stein neighborhoods $U\subset V$ of $\xi_{0}$ such that $H^{r}(U\cap X'_{c}, {\mathcal{F}})=0$ for $1\leq r<n-q$ and
$H^{0}(U, {\mathcal{F}})\rightarrow H^{0}(U\cap X'_{c}, {\mathcal{F}})$ is an isomorphism, or equivalently (See ~\cite{ref6} or
~\cite{ref1}), $\underline{H^{r}_{S}}({\mathcal{F}})=0$ for $r\leq
n-q,$ where
$\underline{H^{r}_{S}}({\mathcal{F}})$ is the cohomology sheaf
with support in $S=\{x\in X: \phi(x)\leq c\}$ and
coefficients in ${\mathcal{F}}.$ Furthermore, there exists a
spectral sequence
$$H^{p}_{S}(X, {\mathcal{F}})\Longleftarrow
E_{2}^{p,q}=H^{p}(X, \underline{H^{q}_{S}}({\mathcal{F}}))$$
Since $\underline{H^{p}_{S}}({\mathcal{F}})=0$ for $p\leq n-q,$ then for any $p\leq n-q,$
the cohomology groups $H^{p}_{S}(X, {\mathcal{F}})$ vanish and, the exact sequence of local
cohomology
\begin{center}
$\cdots\rightarrow H^{p}_{S}(X, {\mathcal{F}}) \rightarrow
H^{p}(X, {\mathcal{F}}) \rightarrow H^{p}(X'_{c}, {\mathcal{F}})
\rightarrow H^{p+1}_{S}(X, {\mathcal{F}})\rightarrow\cdots$
\end{center}
implies that $H^{p}(X, {\mathcal{F}}) \rightarrow H^{p}(X'_{c}, {\mathcal{F}})$
is bijective for any $c\in\mathbb{R}$ if $p\leq n-q-1$ and,
injective if $p=n-q$.
\end{proof}
\begin{lm}{Let $D$ be an open set in $\mathbb{C}^{n}$, $f\in F_{q}(D)$,
$1\leq q\leq n-2$. Then there exists for each point $\xi_{0}\in D$
a fundamental system of Stein neighborhoods $V$ of $\xi_{0}$ such that if $Y=\{z\in D : f(z)>f(\xi_{0})\}$, then for any coherent analytic sheaf ${\mathcal{F}}$ on $D$ we have :\\
(i) $H^{0}(V, \mathcal{F})\rightarrow H^{0}(V\cap Y, \mathcal{F})$ is bijective; \\
(ii) $H^{r}(V\cap Y, \mathcal{F})=0$ for $0<r<n-q$.}
\end{lm}
\begin{proof}
Let $U$ be a Stein neighborhood of $\xi_{0}$ in $D$ such that there exist
finitely many $q$-convex functions $\phi_{1},\cdots, \phi_{s} : U\rightarrow \mathbb{R}$
with $f|_{U}=max(\phi_{1},\cdots, \phi_{s})$.\\
\hspace*{.1in}By lemma $1$, we can choose a Stein open neighborhood $V\subset U$ of $\xi_{0}$ so that
assertions (i) and (ii) are satisfied when $s=1$.\\
\hspace*{.1in}We, obviously, also may assume that the restriction
$f|_{U} : U\rightarrow \mathbb{R}$ is the maximum of two $q$-convex functions
$f|_{U}=max(\phi_{1},\phi_{2})$, which implies that $Y\cap U=Y_{1}\cup Y_{2},$
where $Y_{i}=\{z\in U : \phi_{i}(z)>f(\xi_{0})\}$ for $i=1, 2$.
If $q+1<n$, then there exists a Stein open neighborhood $V\subset U$ of $\xi_{0}$
in $D$ such that 
$H^{0}(V, \mathcal{F})\cong H^{0}(V\cap Y, \mathcal{F})$,
$H^{1}(V, \mathcal{F})\cong H^{1}(V\cap Y, \mathcal{F})=0$.
In fact, there exist by lemma $1$ Stein open neighborhoods $U_{i}\subset U$ of $\xi_{0}$ in $D$,
$i=1, 2$, such that $H^{0}(U_{i}, \mathcal{F})\cong H^{0}(Y_{i}\cap U_{i}, \mathcal{F})$,
$H^{1}(Y_{i}\cap U_{i}, \mathcal{F})\cong H^{1}(U_{i}, \mathcal{F})=0$. We set $V=U_{1}\cap U_{2}$.
If for $j\geq 1$
the open set $Z_{j}=\{z\in Y_{1}\cap V : \phi_{2}(z)>f(\xi_{0})-\frac{1}{j}\}$ is not empty,
then, according to lemma $2$, the restriction maps
$$H^{p}(V, \mathcal{F})\rightarrow H^{p}(Y_{1}\cap V, \mathcal{F}) \ \  \text{and} \ \ H^{p}(Y_{1}\cap V, \mathcal{F})\rightarrow H^{p}(Z_{j}, \mathcal{F})$$
are bijective for $p=0, 1$. Therefore by  Mittag-Leffler theorem one obtains 
$$H^{p}(Y_{1}\cap V, \mathcal{F})=H^p\left(\lim _{\longleftarrow} Z_j, \mathcal{F}\right) \cong \lim _{\longleftarrow} H^p\left(Z_j, \mathcal{F}\right) \cong H^p(Y_{1} \cap Y_{2}\cap V, \mathcal{F}) \ \ \text{for} \ \ p=0, 1$$
We infer that
$H^{0}(Y_{i}\cap V, \mathcal{F})\cong H^{0}(Y_{1}\cap Y_{2}\cap V, \mathcal{F})$ 
and $H^{1}(Y_{i}\cap V, \mathcal{F})\cong H^{1}(Y_{1}\cap Y_{2}\cap V, \mathcal{F})=0$
It follows from the Mayer-Vietoris sequence for cohomology
\begin{center}
$0\rightarrow H^{0}(V\cap Y, \mathcal{F})\rightarrow H^{0}(Y_{1}\cap V, \mathcal{F})\oplus
H^{0}(Y_{2}\cap V, \mathcal{F})\rightarrow H^{0}(Y_{1}\cap Y_{2}\cap V, \mathcal{F})\rightarrow
H^{1}(V\cap Y, \mathcal{F})\rightarrow 0$
\end{center}
that $H^{1}(V\cap Y, \mathcal{F})=0$ and
$H^{0}(V, \mathcal{F})\cong H^{0}(V\cap Y, \mathcal{F})$.\\
\hspace*{.1in}Now if $2\leq r<n-q$, then, a proof similar to the one used previously shows that there exists a Stein neighborhood $V\subset U$
of $\xi_{0}$ such that $H^{r-1}(Y_{i}\cap V, \mathcal{F})\cong H^{r}(Y_{i}\cap V, \mathcal{F})=0$ for $i=1, 2$
and
$H^{r-1}(Y_{1}\cap Y_{2}\cap V, \mathcal{F})\cong H^{r-1}(Y_{1}\cap V, \mathcal{F})=0$, then the Mayer-Vietoris sequence for cohomology
\begin{center}
$\cdots\rightarrow H^{r-1}(Y_{1}\cap V, \mathcal{F})\oplus
H^{r-1}(Y_{2}\cap V, \mathcal{F})\rightarrow H^{r-1}(Y_{1}\cap Y_{2}\cap V, \mathcal{F})\rightarrow
H^{r}(V\cap Y, \mathcal{F})\rightarrow
H^{r}(Y_{1}\cap V, \mathcal{F})\oplus
H^{r}(Y_{2}\cap V, \mathcal{F})\rightarrow\cdots$
\end{center}
implies that  $H^{r}(V\cap Y, \mathcal{F})=0$.
\end{proof}

\begin{lm}{Let $X$ be a complex manifold of dimension $n$, and let $f:
X\rightarrow \mathbb{R}$ be a continuous $q$-convex with corners function on $X$. Let $\xi_{0}\in X$
and $X'_{c}=\{x\in X: f(x)>c\},$ where $c=f(\xi_{0})$.
Then for any coherent analytic sheaf ${\mathcal{F}}$
on $X$ the restriction map
$$H^{p}(X,{\mathcal{F}})\rightarrow H^{p}(X'_{c}, {\mathcal{F}})$$
is bijective if $p\leq n-q-1$,\\
\hspace*{.1in}injective if $p=n-q.$}
\end{lm}
\begin{proof}
Let $V\subset\subset X$ be an open neighborhood of $\xi_{0}$ biholomorphic to a domain in $\mathbb{C}^{n}$.
Then there exists, by lemma $2$, a fundamental system of connected Stein neighborhoods $U\subset V$ of $\xi_{0}$ such that $H^{r}(U\cap X'_{c}, {\mathcal{F}})=0$ for $1\leq r<n-q$ and
$H^{0}(U, {\mathcal{F}})\rightarrow H^{0}(U\cap X'_{c}, {\mathcal{F}})$ is an isomorphism. Then a similar proof as that of lemma $2$ 
shows that the cohomology sheaf $\underline{H^{r}_{S}}({\mathcal{F}})$ with support in $S=\{x\in X: f(x)\leq c\}$ and
coefficients in ${\mathcal{F}}$ vanishes for $r\leq n-q$.
Therefore, by using the spectral sequence
$$H^{p}_{S}(X, {\mathcal{F}})\Longleftarrow
E_{2}^{p,q}=H^{p}(X, \underline{H^{q}_{S}}({\mathcal{F}}))$$
we find that the cohomology group $H^{p}_{S}(X, {\mathcal{F}})=0$ for $p\leq n-q$.
and, the exact sequence of local
cohomology
\begin{center}
$\cdots\rightarrow H^{p}_{S}(X, {\mathcal{F}}) \rightarrow
H^{p}(X, {\mathcal{F}}) \rightarrow H^{p}(X'_{c}, {\mathcal{F}})
\rightarrow H^{p+1}_{S}(X, {\mathcal{F}})\rightarrow\cdots$
\end{center}
implies that $H^{p}(X, {\mathcal{F}}) \rightarrow H^{p}(X'_{c}, {\mathcal{F}})$
is bijective if $p\leq n-q-1$ and,
injective if $p=n-q$.
\end{proof}
\begin{center}
\bf{Proof of the theorem.}
\end{center}
\begin{proof}
Let $V\subset\subset X$ be an open neighborhood of  $\xi_{0}$ that can be identified to a domain in $\mathbb{C}^{n}$.
Then there exists, by lemma $1$, a fundamental system of connected Stein neighborhoods $U\subset V$ of $\xi_{0}$ such that 
for every ${\mathcal{F}}\in coh(X)$, $H^{p}(U\cap Y, {\mathcal{F}})=0$ for $p=1, 2$ and
$H^{0}(U, {\mathcal{F}})\rightarrow H^{0}(U\cap Y, {\mathcal{F}})$ is an isomorphism.\\
\hspace*{.1in}We may choose the open sets $U$ to be open balls centered at $\xi$ with radius $r \to 0$. Then $U$ is contractible. Therefore $H^{p}(U,\mathbb{Z})=0$ for $p\geq 1$ and from the long exact sequence of cohomology associated to the short exact sequence of sheaves
$$
0 \;\longrightarrow\; \mathbb{Z} \;\xrightarrow{\;\times 2i\pi\;}\; \mathcal{O} \;\xrightarrow{\;\exp\;}\; \mathcal{O}^{*} \;\longrightarrow\; 0,
$$
it follows that $H^{p}(U,\mathcal{O}^{*}) = 0$ for $p\geq 1$.\\
\hspace*{.1in}It should be noted first that if $L$ is a trivial holomorphic line bundle on $U\cap Y$, then it is clear that $L$ is analytically trivial.
In fact, since the restriction map 
$$H^{p}(U, {\mathcal{O}})\rightarrow H^{p}(U\cap Y, {\mathcal{O}})$$ 
is bijective for $p=0,1, 2$ and $U$
is connected, then $H^{p}(U\cap Y,{\mathcal{O}})=0$ for $p=1,2$ and the restriction map
$$H^{0}(U, {\mathcal{M}^{*}})\longrightarrow H^{0}(U\cap Y, {\mathcal{M}^{*}})$$ 
is also bijective.\\
\hspace*{.1in}Let $\mathcal{D}_{U\cap Y}$ denote the sheaf of Cartier divisors on $U\cap Y$, that is $\mathcal{D}_{U\cap Y}=\mathcal{M}_{U\cap Y}^{\star} / \mathcal{O}_{U\cap Y}^{\star}$. 
Then there exists an exact sequence of sheaves on $U\cap Y$ :
$$
1 \longrightarrow \mathcal{O}_{U\cap Y}^{\star} \longrightarrow \mathcal{M}_{U\cap Y}^{\star} \longrightarrow \mathcal{D}_{U\cap Y} \longrightarrow 0
$$
and the associated exact sequence of cohomology 
$$\cdots\rightarrow H^{0}(U\cap Y, {\mathcal{D}}_{U\cap Y})\rightarrow H^{1}(U\cap Y, {\mathcal{O}^{*}_{U\cap Y}})\rightarrow H^{1}(U\cap Y, {\mathcal{M}}^{*})\rightarrow\cdots $$
Since $H^{1}(U\cap Y,{\mathcal{O}})=0$, then from the exact sequence of cohomology
$$\begin{aligned}
%&\text { induces in cohomology an exact sequence }\\
&H^1(U\cap Y, \mathcal{O}) \longrightarrow H^1\left(U\cap Y, \mathcal{O}^*\right) \longrightarrow H^2(U\cap Y, \mathbb{Z}) .
\end{aligned}$$
it follows that the set of topologically trivial holomorphic line bundles 
$$\operatorname{Pic}^0(U\cap Y)= Ker(H^1\left(U\cap Y, \mathcal{O}^*\right) \longrightarrow H^2(U\cap Y, \mathbb{Z}))$$ 
is contained in $Im(\mathcal{D}(U\cap Y)\stackrel{\delta_{U\cap Y}}\longrightarrow Pic(U\cap Y)\cong H^{1}(U\cap Y, {\mathcal{O}}^{*})$).\\
Therefore, if $L$ is a trivial holomorphic line bundle over $Y\cap U$, 
since $H^1(U\cap Y, \mathcal{O})=0$,  then $L$ is associated to some Cartier divisor $D$ on $U\cap Y$.
It is known that a holomorphic line bundle $L$ over a connected complex manifold $M$ is associated to a Cartier divisor if and only if
$L$ admits a non-trivial meromorphic section. Then there exists $\sigma\in H^{0}(U\cap Y, {\mathcal{M}^{*}})$
such that $L=O(div(\sigma))$. Since $H^0\left(U, \mathcal{O}\right)\cong H^0\left(U\cap Y, \mathcal{O}\right)$ and $U$ is connected, then the restriction map 
$$H^0\left(U, \mathcal{M}^*\right) \longrightarrow \left(U \cap Y, \mathcal{M}^*\right)$$ 
is bijective, therefore there exists a unique $\tilde{\sigma}\in H^{0}(U, \mathcal{M}^{*})$ such that
$\tilde{\sigma}|_{U\cap Y}=\sigma$. Since by assumption $\tilde{L}=O(div(\tilde{\sigma}))$ is analytically trivial, then obviously 
the class $[L]=[O(div(\tilde{\sigma}|_{U\cap Y})]=[\tilde{L}|_{U\cap Y}]=0$.\\
\hspace*{.1in}We are now going to prove that
$H^{1}(U\cap Y, \mathcal{O}^{*})=0$. 
Since the restriction map $\Gamma(U,\mathcal{O})\rightarrow \Gamma(U\cap Y,\mathcal{O})$ is bijective,
then $\Gamma(U,\mathcal{O}^{*})\rightarrow \Gamma(U\cap Y,\mathcal{O}^{*})$ is also bijective.\\
It is sufficient, according to the exact sequence of local cohomology :
$$0=H^{1}(U,{\mathcal{O}^{*}})\rightarrow H^{1}(U\cap Y,{\mathcal{O}^{*}})\rightarrow H^{2}_{S}(U,{\mathcal{O}}^{*})\rightarrow H^{2}(U,{\mathcal{O}}^{*})=0,$$
where $S=U\setminus Y$, to show that $H^{2}_{S}(U,{\mathcal{O}}^{*})=0$.
For this, we consider the exact sequence of local cohomology :
$$\cdots\rightarrow H^{2}_{S}(U,{\mathcal{O}})\rightarrow H^{2}_{S}(U,{\mathcal{O}}^{*})\rightarrow H^{3}_{S}(U,\mathbb{Z})\rightarrow\cdots$$
Since $H^{2}(U, {\mathcal{O}})=H^{1}(U\cap Y, {\mathcal{O}})=0$, then $H^{2}_{S}(U,{\mathcal{O}})=0$.
It remains to prove that $H^{3}_{S}(U,\mathbb{Z})=0$. To see this, we consider the commutative diagram with exact rows
$$\begin{array}{cccccc}
H^{1}(U,\mathcal O^*) 
  & \xrightarrow{\alpha} 
  & H^2(U,\mathbb Z) 
  & \xrightarrow{} 
  & H^2(U,\mathcal O)=0 \\[6pt]
\downarrow r_1 
  & 
  & \downarrow r_2 
  & 
  & \downarrow \\[6pt]
H^{1}(U\cap Y,\mathcal O^*) 
  & \xrightarrow{\beta} 
  & H^2(U\cap Y,\mathbb Z) 
  & \xrightarrow{} 
  & H^2(U\cap Y,\mathcal O)=0
\end{array}$$
%\]
Since $r_{1}$ is an isomorphism and $\beta$ is surjective, it follows that $r_{2}$ is surjective.
Then by considering the exact sequence 
$$\cdots\rightarrow H^{2}(U,\mathbb{Z})\stackrel{r_{2}}\rightarrow H^{2}(U\cap Y,\mathbb{Z})\rightarrow H^{3}_{S}(U, \mathbb{Z})\rightarrow H^{3}(U, \mathbb{Z})$$
and noting that $r_{2}$ is surjective and $H^{3}(U, \mathbb{Z})=0$, we find that $H^{3}_{S}(U, \mathbb{Z})=0$ and therefore
$H^{1}(U\cap Y, \mathcal{O}^{*})=0$.\\
\hspace*{.1in}Now since $H^{0}(U, {\mathcal{O}^{*}})\rightarrow H^{0}(U\cap Y, {\mathcal{O}^{*}})$ is bijective
and $H^{p}(U, \mathcal{O}^{*})=H^{p}(U\cap Y, \mathcal{O}^{*})=0$ for $p=1, 2$, it follows from
~\cite{ref6} that $\underline{H^{p}_{S}}({\mathcal{O}^{*}})=0$ for $p\leq
3,$ where
$\underline{H^{p}_{S}}({\mathcal{O}^{*}})$ is the cohomology sheaf
with support in $S=\{z\in X: \phi(z)\leq c\}$ and
coefficients in ${\mathcal{O}^{*}}.$ We deduce from the spectral sequence
$$H^{p}_{S}(X, {\mathcal{O}^{*}})\Longleftarrow
E_{2}^{p,q}=H^{p}(X, \underline{H^{q}_{S}}({\mathcal{O}^{*}}))$$ 
that the cohomology group $H^{p}_{S}(X, {\mathcal{O}^{*}})$ vanishes for any $0\leq p\leq 3$ and, the exact sequence of local
cohomology
\begin{center}
$\cdots\rightarrow H^{p}_{S}(X, {\mathcal{O}^{*}}) \rightarrow
H^{p}(X, {\mathcal{O}^{*}}) \rightarrow H^{p}(Y, {\mathcal{O}^{*}})
\rightarrow H^{p+1}_{S}(X, {\mathcal{O}^{*}})\rightarrow\cdots$
\end{center}
implies that $H^{p}(X, {\mathcal{O}^{*}}) \rightarrow H^{p}(Y, {\mathcal{O}^{*}})$
is bijective if $0\leq p\leq 2$,\\ 
and injective if $p=3$.
\end{proof}

\end{document}